\documentclass[12pt,twoside]{article}
\usepackage{amsfonts}
\usepackage{cite}
 \usepackage{color}
\usepackage{amsmath} \numberwithin{equation}{section}
\usepackage{amsthm}
 \usepackage{graphicx}
\usepackage{mathrsfs}

\topmargin by-\headsep  \textheight 24cm
\newtheorem{lemma}{Lemma}[section]

\newtheorem{definition}{Definition}[section]
\newtheorem{theorem}[definition]{Theorem}

\begin{document}
\title{A Note on the Extinction of a Stochastic Differential Equation  SIS Epidemic Model}%
\author{Wenshu Zhou$^{1,2}$\quad Xu Zhao$^{2}$ \quad Xiaodan Wei$^{3}$\thanks{Corresponding author. \newline E-mail address:   weixiaodancat@126.com}
 \quad  \\
  \small  1. Department of Mathematics, Dalian Minzu University, Dalian
 116600,  P. R. China\\
 \small 2. School of Mathematics, Beifang Minzu University,  Yinchuan 750021, P. R. China\\
 \small 3.  College of Computer  Science, Dalian
Minzu University, Dalian  116600,  P. R. China\\}
 \pagestyle{myheadings}
\markboth{ }%
{}

\date{}
 \maketitle

\begin{abstract}
The aim of this note is to give a new proof to the conjecture, proposed by Gray et al. (2011), on the extinction of a stochastic differential equation SIS epidemic model. The conjecture was proved by  Xu (2017). Our proof is much more direct and simpler.



 \end{abstract}


%

\section{Introduction}
Epidemic models are inevitably affected by environmental white noise which is an important component in realism. In fact, the presence of a noise source   can modify the behavior of corresponding deterministic evolution of the
system (cf.\cite{Gard,Mao,TBV,DGM,GGHM}). For this reason, many stochastic epidemic models have been developed, see for instance \cite{JJ,LO,YM,LS,ZJ1,LJ,ZHXM,ZJMG,Zhao,WL,LJS} and references therein.
 However, compared to deterministic systems, it is extremely difficult to give the basic reproduction number of stochastic systems. Recently, Gray et al. \cite{GGHM} studied the following stochastic SIS system:
\begin{equation}\label{e0}
\left\{\begin{split}
& d S(t)=[\mu N-\beta S(t)I(t)+\gamma I(t)-\mu S(t)]dt-\sigma S(t)I(t)dB(t), \\
& d I(t) = [\beta S(t)I(t)-(\mu+\gamma)I(t)]dt+\sigma S(t)I(t)dB(t).
\end{split}\right.
\end{equation}
Here $S$ and $I$ denote the susceptible and infected fractions of the population, respectively. $\mu$ is the per capita death rate, and $\gamma$ is the rate at which
infected individuals become cured,   $\beta$ is the
disease transmission coefficient.  $B(t)$ is the Brownian motion with $B(0)=0$ and with   the intensity  $\sigma^2>0$ . Given that $S(t) + I(t) \equiv N$,  they simplified system \eqref{e0} into the following single equation:
\begin{equation}\label{e01}
 \begin{split}
 d I(t) = [(\beta N-\mu-\gamma)I(t)-\beta I^2(t)]dt+\sigma(N-I(t))I(t)dB(t).
\end{split}
\end{equation}
They firstly showed that for any given initial value $I(0)=I_0 \in (0,N)$,   system \eqref{e01}
has a unique global solution $I(t) \in (0,N)$ for all $t \geq 0$ with probability one. Furthermore, they studied the extinction and the persistence of system \eqref{e01}, and proved that if either
\begin{equation}
\begin{split}
(i)~~R_0^S<1~\hbox{and}~ \sigma^2 \leq \frac{\beta}{N}~~\hbox{or}~~(ii)~~\sigma^2>\frac{\beta}{N}\vee\frac{\beta^2}{2(\mu+\beta)}, \end{split}
\end{equation}
where $R_0^S=\frac{\beta N}{\mu+\gamma}-\frac{\sigma^2N^2}{2(\mu+\gamma)}$, then  the disease  will die  out, whereas if $R^S_0 >1$, then the
disease will persist. Naturally, they proposed the following conjecture (see Conjecture 8.1 in \cite{GGHM}):

\indent{\bf Conjecture}~~If $R_0^S<1$ and $\frac{\beta^2}{2(\mu+\gamma)}\geq\sigma^2 >\frac{\beta}{N}$, then the disease will die out.

\noindent{Recently, Xu \cite{Xu}} proved this conjecture  by using Feller's test (cf. Appendix A in \cite{Xu}). The aim of the present paper is to give a new proof of the  conjecture, which   is much more direct and simpler.  Our main result is as follows.

\begin{theorem} Assume that $I(t)$ is the solution of system \eqref{e01} with  the initial value
$I(0)=I_0 \in (0, N)$. We have

(i)~~If either $\frac{\sigma^2}{2}\left(N+\frac{\mu+\gamma}{\beta}\right)\leq \beta$  or  $N \leq \frac{\mu+\gamma}{\beta}$, then
\begin{equation*}
\begin{split}
\mathop{\lim\sup}\limits_{t\rightarrow+\infty}\frac{\log I(t)}{t}\leq(\mu+\gamma)(R_0^S-1)~~a.s.
\end{split}
\end{equation*}

(ii)~~If $\frac{\sigma^2}{2}\left(N+\frac{\mu+\gamma}{\beta}\right)> \beta$ and $N>\frac{\mu+\gamma}{\beta}$, then
\begin{equation*}
\begin{split}
\mathop{\lim\sup}\limits_{t\rightarrow+\infty}\frac{\log I(t)}{t}\leq-\frac{\sigma^2}{2}\left(\frac{\mu+\gamma}{\beta}\right)^2~~a.s.
\end{split}
\end{equation*}
Namely, $I(t)$ goes to zero exponentially a.s. if $R_0^S<1$. In other words, the disease will  die   out with
probability one.
\end{theorem}

\section{Proof of  Theorem 1.1}

For convenience we introduce the notation: $\langle x(t)\rangle=\frac{1}{t}\int_0^tx(s)ds.$

The following lemma will play an important role in the proof of Theorem 1.1.

\begin{lemma}Assume that $I(t)$ is the solution of system \eqref{e01} with the initial value
$I(0)=I_0\in (0, N)$. We have

(i)~~If $N\leq \frac{\mu+\gamma}{\beta}$, then $
\lim\limits_{t\rightarrow+\infty}\langle I(t) \rangle=0~~a.s.
$

(ii)~~If $N>\frac{\mu+\gamma}{\beta}$, then
$
\mathop{\lim\sup}\limits_{t\rightarrow+\infty}\langle I(t) \rangle\leq N-\frac{\mu+\gamma}{\beta}~~a.s.
$
\end{lemma}

\begin{proof}Integrating \eqref{e01} from $0$ to $t$ yields
\begin{equation}\label{e061}
\begin{split}I(t)-I(0)=&(\beta N-\mu-\gamma)\int_0^t I(s)ds-\beta\int_0^t I^2(s)ds\\
&+\sigma \int_0^t (N-I(s))I(s)dB(s),
\end{split}
\end{equation}
and dividing $t$ on both sides of \eqref{e061} gives
\begin{equation}\label{e03}
\begin{split}\langle I^2(t)\rangle=\left(N-\frac{\mu+\gamma}{\beta}\right)\langle I(t)\rangle+\psi(t),
\end{split}
\end{equation}
where $\psi(t)=\frac{I(0)-I(t)}{\beta t}+\frac{\sigma}{\beta t} \int_0^t (N-I(s))I(s)dB(s)$. By the H\"{o}lder inequality, we have
\begin{equation}\label{e06}
\begin{split}\langle I(t)\rangle^2\leq\left(N-\frac{\mu+\gamma}{\beta}\right)\langle I(t)\rangle+\psi(t).
\end{split}
\end{equation}
By the large number theorem for martingales (cf. \cite{GGHM,Mao}), we have
 \begin{equation*}
\begin{split}\lim\limits_{t\rightarrow+\infty} \frac{\sigma}{\beta t}\int_0^t (N-I(s))I(s)dB(s)=0~~a.s.
\end{split}
\end{equation*}
This  leads to
\begin{equation}\label{e451}
\begin{split}\lim\limits_{t\rightarrow+\infty}\psi(t)=0~~a.s.
\end{split}
\end{equation}

If $N\leq\frac{\mu+\gamma}{\beta}$, then the conclusion (i) follows from \eqref{e06} and \eqref{e451} immediately.

If $N>\frac{\mu+\gamma}{\beta}$, we deduce from \eqref{e06} and \eqref{e451} that
\begin{equation*}\begin{split}
\Big(\mathop{\lim\sup}\limits_{t\rightarrow+\infty}\langle I(t)\rangle\Big)^2\leq\left(N-\frac{\mu+\gamma}{\beta}\right)\mathop{\lim\sup}\limits_{t\rightarrow+\infty}\langle I(t)\rangle~~a.s.
\end{split}
\end{equation*}
This implies the  conclusion (ii), and ends the proof of Lemma 2.1.
\end{proof}

Now we can give the proof of Theorem 1.1.

\noindent{\bf Proof of Theorem 1.1}~~By the It\^{o} formula, we deduce from \eqref{e01}  that
\begin{equation}\label{e04}
\begin{split}d\log I(t)=& \beta N-\mu-\gamma -\beta I(t)-\frac{1}{2}\sigma^2(N-I(t))^2+\sigma (N-I(t))dB(t)\\
=&\beta N-\mu-\gamma-\frac{1}{2}\sigma^2 N^2+(\sigma^2 N-\beta)I(t)-\frac12\sigma^2I^2(t)+\sigma (N-I(t))dB(t).
\end{split}
\end{equation}
Integrating \eqref{e04} from $0$ to $t$ and dividing $t$ on  both sides of the resulting equation, we obtain
\begin{equation}\label{e05}
\begin{split}\frac{\log I(t)-\log I(0)}{t}=&  \beta N-\mu-\gamma-\frac{1}{2}\sigma^2 N^2+(\sigma^2 N-\beta)\langle I(t)\rangle-\frac12\sigma^2\langle I^2(t)\rangle\\
&+\frac{\sigma}{t}\int_0^t (N-I(s))dB(s).
\end{split}
\end{equation}
Substituting \eqref{e03} into \eqref{e05} yields
\begin{equation}\label{e415}
\begin{split}\frac{\log I(t)}{t}=\left(\beta N-\mu-\gamma-\frac{1}{2}\sigma^2N^2\right)+\left[\frac{\sigma^2}{2}\left(N+\frac{\mu+\gamma}{\beta}\right)-\beta\right]\langle I(t)\rangle+ \Psi(t),
\end{split}
\end{equation}
where $\Psi(t)=\frac{\log I(0)}{t}+\frac{\sigma}{t}\int_0^t (N-I(s))dB(s)-\frac{\sigma^2}{2}\psi(t).$ By the large number theorem for martingales (cf. \cite{GGHM,Mao}), we have
 \begin{equation*}
\begin{split}\lim\limits_{t\rightarrow+\infty} \frac{\sigma}{t}\int_0^t  (N-I(s))dB(s)=0~~a.s.
\end{split}
\end{equation*}
This, together with \eqref{e451}, leads to
\begin{equation}\label{e07}
\begin{split}\lim\limits_{t\rightarrow+\infty} \Psi(t) =0~~a.s.
\end{split}
\end{equation}
If either $\frac{\sigma^2}{2}\left(N+\frac{\mu+\gamma}{\beta}\right)\leq \beta$ or $N\leq \frac{\mu+\gamma}{\beta}$, by Lemma 2.1 (i) and \eqref{e07},  we derive from  \eqref{e415} that
\begin{equation*}
\begin{split}\mathop{\lim\sup}\limits_{t\rightarrow+\infty}\frac{\log I(t)}{t}\leq\beta N-\mu-\gamma-\frac{1}{2}\sigma^2N^2 =(\mu+\gamma)(R_0^S-1)~~a.s.
\end{split}
\end{equation*}
 If $\frac{\sigma^2}{2}\left(N+\frac{\mu+\gamma}{\beta}\right)>\beta$ and $N>\frac{\mu+\gamma}{\beta}$,   by  \eqref{e07} and Lemma 2.1 (ii),  we derive from  \eqref{e415} that
\begin{equation*}
\begin{split}
\mathop{\lim\sup}\limits_{t\rightarrow+\infty}\frac{\log I(t)}{t}&=\left(\beta N-\mu-\gamma-\frac{1}{2}\sigma^2N^2\right)+\left[\frac{\sigma^2}{2}\left(N+\frac{\mu+\gamma}{\beta}\right)-\beta\right]\mathop{\lim\sup}\limits_{t\rightarrow+\infty}\langle I(t)\rangle\\
&\leq \left(\beta N-\mu-\gamma-\frac{1}{2}\sigma^2N^2\right)+\left[\frac{\sigma^2}{2}\left(N+\frac{\mu+\gamma}{\beta}\right)-\beta\right] \left(N-\frac{\mu+\gamma}{\beta}\right)\\
&=-\frac{\sigma^2}{2}\left(\frac{\mu+\gamma}{\beta}\right)^2~~a.s.
\end{split}
\end{equation*}
The proof of Theorem 1.1 is completed.

\vskip0.4cm
\section*{Acknowledgments}
The research  was  supported  by  the NSFC (11571062),  the Program for Liaoning
Innovative Talents in University (LR2016004),  and  the Fundamental Research Fund  for the Central Universities (DMU).
\par

\end{document}